\documentclass[10pt,final,3p,times]{elsarticle}
\journal{FILOMAT}
\usepackage{amsmath}
\usepackage{lipsum}
\usepackage{amsfonts}
\usepackage{graphicx}
\usepackage{natbib}
\usepackage{hyperref}
\usepackage{epsfig}
\usepackage{amssymb}
\usepackage{cite}
\usepackage{float}
\usepackage{multirow}
\usepackage{subcaption}
\usepackage{balance}

\usepackage{amsthm}
 \newtheorem{thm}{Theorem}[section]
 \newtheorem{cor}[thm]{Corollary}

 \newtheorem{defn}[thm]{Definition}
 \newtheorem{exmp}{Example}[section]
\begin{document}
\begin{frontmatter}
\title{A rapidly convergent approximation scheme for nonlinear autonomous and non-autonomous wave-like equations 
}
\author{Prakash Kumar Das$^1$}
\ead{prakashdas.das1@gmail.com}

\author{M.M. Panja$^2$\corref{mycorrespondingauthor}}
\ead{madanpanja2005@yahoo.co.in}

\address{$^1$Department of Mathematics, Trivenidevi Bhalotia College, Raniganj-713 347, Burdwan,
West Bengal, India \\
$^2$Department of Mathematics, Visva-Bharati, Santiniketan - 731 235, West Bengal, India
}

\begin{abstract}
In this work, an efficient approximation scheme has been proposed for getting accurate approximate solution of nonlinear partial differential equations with constant or variable coefficients satisfying initial conditions in a series of exponential instead of an algebraic function of independent variables. As a consequence: i) the convergence of the series found to be faster than the same obtained by few other methods and ii) the exact analytic solution can be obtained from the first few terms of the series of the approximate solution, in cases the equation is integrable. The convergence of the sum of the successive correction terms has been established and an estimate of the error in the approximation has also been presented. The efficiency of the present method has been illustrated through some examples with a variety of nonlinear terms present in the equation.
\end{abstract}
\end{frontmatter}
\section{Introduction}
In many branches of physical, biological and engineering sciences, mathematical analysis of natural phenomena are found to begin with nonlinear ordinary/partial differential equations(ODEs/ PDEs). Solutions of such equations help one to realize the process. But apart from some ideal cases, getting the exact solution of mathematical models of natural processes comprising ODEs or PDEs with an appropriate initial/boundary conditions is a formidable task. So, the development of an efficient method for getting an exact solution in a compact form or reliable approximate solution of such equations is highly desirable.

A few analytical methods such as method based on symmetry analysis (Lie group theoretic approach) \citep{olver2000applications,ibragimovcrc}, Prelle-Singer method \citep{choudhury2008solutions}, the method involving Jacobi last multiplier \citep{nucci2005jacobi}, Tanh, Sech, Exp method \citep{malfliet1992solitary, malfliet1996tanh} and so on, analytical approximation schemes such as  homotopy analysis method (HAM)\citep{abbasbandy2009homotopy,xinhui2012homotopy}, Adomian decomposition method (ADM) \citep{adomian1994solving}, Fourier transform Adomian decomposition method (FTADM) \citep{nourazar2013new}, rapidly convergent approximation method (RCAM) \citep{das2015improved,das2016rapidly,das2018rapidly,das2018solutions,das2018piecewise,das2019some,das2019rapidly} etc., numerical methods viz. finite difference/element methods \citep{butcher2003numerical,brenner2007mathematical}, Galerkin or collocation methods are used to find the solution of mathematical models. In every method, the main interest is to develop a scheme that yields a physically meaningful solution with a minimum computational effort. The effort expended can be assessed in terms of the straightforward implementation of the method and computer resources required to achieve a specified accuracy. 

Among the approximation methods mentioned above, ADM is found to be simpler. In ADM, a solution corresponding to the highest order derivative term only (regarded as the unperturbed or leading part) has been taken as the leading order solution. Successive corrections are obtained by a recursive scheme based on the application of the inverse of the unperturbed part of the operator on the correction terms of the previous order. The rate of convergence of the correction terms is found relatively slow in general. Naturally, it is of curiosity to verify whether a straightforward method can be developed which can provide either an exact solution in a compact form or a rapidly convergent approximate solution to the problem in an efficient  way.

In this work, we have addressed this problem and proposed a recursive scheme for getting an exact solution or rapidly convergent approximate solution of nonlinear initial value problems involving PDEs with constant or variable space dependent coefficients. Here we have introduced an operator associated with the (whole) linear part of the equation and derived a straightforward formula involving inverse of such operator for correction terms associated with the nonlinear part of the differential equation involved in the model. It provides the solution in a series of exponentials instead of power of independent variables as appear in the case of conventional ADM. 

The present investigation has been described in six sections as follows. Salient steps of ADM have been discussed in sections \ref{sec2} and its generalization (rapidly convergent approximate scheme (RCAS)) towards a scalar autonomous and non-autonomous nonlinear partial differential equation has been presented in section \ref{sec3}. An estimate of error has been analysed in section \ref{sec4}. The usefulness and efficiency of the scheme developed here have been illustrated through a few examples in section \ref{sec5}. Salient features of the proposed scheme have been summarized in section \ref{sec6}.

\section{ADM for nonlinear wave-like equation}\label{sec2}
Here we describe the salient steps of ADM for nonlinear wave-like equation
\begin{equation}\label{eq2p1}
u_{tt}(X,t)= {\cal N}[u](X,t)+\lambda^2 u(X,t)+S(X,t)
\end{equation}
with initial conditions
\begin{equation}\label{eq2p2}
u(X,0) = a_0(X), \ u_t(X,0) = a_1(X)
\end{equation}
Here $X=(x_1,x_2,\cdots,x_n) \in \Omega' \subseteq \mathbb{R}^n$,$\,S(X,t)$ is a given analytic source term and ${\cal N}[u](X,t)$ is nonlinear term involving independent variables $X,t$ and the dependent variable $u$.
It is customary to recast Eq.(\ref{eq2p1}) in the operator form
\begin{equation}\label{eq2p3}
{\cal L}_{tt}[u](X,t)= {
\cal N}[u](X,t)+\lambda^2 u(X,t)+S(X,t)
\end{equation}
where ${\cal L}_{tt}[\cdot]=\frac{\partial^2}{\partial t^2}[\cdot]$. Then the inverse operator ${\cal L}_{tt}^{-1}[\cdot]=\int_0^t \left\{\int_0^{t'}[\cdot](t'') \ dt''\right\} \ dt'$ seems to exist. Application of this inverse operator into both sides of Eq.(\ref{eq2p3}) gives
\begin{equation}\label{eq2p4}
u(X,t)= u(X,0)+t u_t(X,0)+{\cal L}_{tt}^{-1}[{
\cal N}[u](X,t)]+\lambda^2 {\cal L}_{tt}^{-1}[u(X,t)]+{\cal L}_{tt}^{-1}[S](X,t).
\end{equation}
In ADM and all its modifications, unknown solution $u(X,t)$ is assumed in the form
\begin{equation}\label{eq2p5}
u(X,t) =  \sum_{k=0}^\infty u_k(X,t)
\end{equation}
and recast the nonlinear term ${\cal N}[u](X,t)$  into in a series of terms
\begin{equation}\label{eq2p6}
{\cal N}[u](X,t)\equiv \sum_{m=0}^\infty {\cal A}_m(X,t) = \sum_{m=0}^\infty {\cal A}_m(u_0(X,t),u_1(X,t),\cdots,u_{m}(X,t)).
\end{equation}
Here ${\cal A}_m (X,t)= {\cal A}_m (u_0(X,t),u_1(X,t),\cdots,u_m(X,t)), m =0,1,\cdots $ are known as Adomian polynomials extracted from nonlinear term by using the formula \citep{adomian1994solving}
\begin{equation}\label{eq2p7}
{\cal A}_m(X,t)= \frac{1}{m!}\left[\frac{d^m}{d\epsilon^m}{\cal N}\left(\sum_{k=0}^\infty u_k\epsilon^k \right)\right]_{\epsilon=0}, \ \ m \geq 0.
\end{equation}
The pedagogical scheme for obtaining the unknown leading component $u_0(X,t)$ and higher order corrections  $u_n(X,t), \ n \geq 1,$ is the following.

The leading approximation has been taken as the solution of ${\cal L}_{tt}[u](X,t)=S(X,t)$ satisfying initial condition (\ref{eq2p2}) so that it can be obtained in a straightforward way from the formula
\begin{equation}\label{eq2p8}
u_0(X,t)= a_0(X)+t\; a_1(X)+{\cal L}_{tt}^{-1}[S](X,t).
\end{equation}
Here the terms involving initial conditions is a first order polynomial in $t$. The successive terms in the expansion (\ref{eq2p5}) for $u(X,t)$ can then be obtained recursively by using the formula
\begin{equation}\label{eq2p9}
u_{n+1}(X,t)= \lambda^{2}{\cal L}_{tt}^{-1}[u_n](X,t)+{\cal L}_{tt}^{-1}[{\cal A}_n](X,t),\ \ \  n\geq 0.
\end{equation}
It is important to mention here that ADM does not include linear part $u(X,t)$ into the operator ${\cal L}_{tt}$. In its revised form proposed here all linear terms involved in the equation have been assembled in the operator ${\cal L}$ to get a more accurate and physically meaningful leading order approximation of the solution to the problem as described in the following section.

\section{RCAS for nonlinear wave-like equation}\label{sec3}
Instead of considering (\ref{eq2p1}), we consider a relatively general form
\begin{equation}\label{eq2p10}
u_{tt}(X,t)-( \lambda_1 + \lambda_2)\; u_{t}(X,t)  + \lambda_1 \lambda_2 \; u(X,t)={\cal N}[u,\cdots,u_{x_i},\cdots,u_{x_i x_j}\cdots,\cdots](X,t) + S(X,t)
\end{equation}
with the same initial condition as given in (\ref{eq2p2}). Here $u_{x_i},u_{x_i x_j}$ etc., represents partial derivatives of $u$ with respect to the variables present in the subscripts.  In contrast to defining $\hat{\cal O}$ involving only the (linear) highest order derivative term only in  the classical ADM, we define instead the operator$ \ \hat{\cal O}[\cdot](X,t) = \frac{\partial^2}{\partial t^2}[\cdot](X,t) -(\lambda_1+\lambda_2)\frac{\partial}{\partial t}[\cdot](X,t)+\lambda_{1}\lambda_{2}[\cdot](X,t) \equiv (\frac{\partial}{\partial t}-\lambda_{2})(\frac{\partial}{\partial t}-\lambda_{1})[\cdot](X,t)$ incorporating all linear terms present in the equation. This helps to recast Eq.(\ref{eq2p10}) into the form
\begin{equation}\label{eq2p11}
\hat{\cal{O}}[u](X,t)={\cal N}[u,\cdots,u_{x_i},\cdots,u_{x_i x_j}\cdots,\cdots](X,t) + S(X,t), \ \ X \in \Omega' \subseteq \mathbb{R}^n.
\end{equation}
It may be useful to mention here that the linear operator $ \hat{{\cal O}}[\cdot] $ can be written in the form
\begin{equation}\label{eq2p12}
\hat{{\cal O}}[\cdot] = e^{\lambda_{2}t}\frac{\partial}{\partial t}\left(e^{(\lambda_{1}-\lambda_{2})t}\frac{\partial}{\partial t}\left(e^{-\lambda_{1}t}[\cdot]\right)\right).
\end{equation}
Reformulation of $\hat{\cal O}$ mentioned above plays a key role in expressing the solution in terms of rapidly convergent series of exponentials. This form of $\hat{{\cal O}}$ immediately provides the inverse operator $\hat{{\cal O}^{-1}}$ as a twofold integral operator given by
\begin{equation}\label{eq2p13}
\hat{{\cal O}}^{-1}[\cdot](X,t)  = e^{\lambda_{1}t}\int_a^t e^{(\lambda_{2}-\lambda_{1})t^{'}}\int_a^{t'}e^{-\lambda_{2}t^{''}}[\cdot](X,t'') \ dt'' \ dt'.
\end{equation}
It's worth it to mention  here that representing inverse of a linear operator with variable coefficients by integrals is also possible whenever it is factorisable.
It may be noted that operation of $\hat{{\cal O}}^{-1}$ mentioned above on $ (u_{tt}(X,t)-(\lambda_1 + \lambda_2)\; u_{t}(X,t)  + \lambda_1 \lambda_2 \; u(X,t))$ leads to
\begin{align}\label{eq2p14}
\hat{{\cal O}}^{-1}& \left[u_{tt}(X,t)-(\lambda_1 + \lambda_2)\; u_{t}(X,t)  + \lambda_1 \lambda_2 \; u(X,t)\right] \nonumber \\
&= e^{\lambda_{1}t}\int_a^t e^{(\lambda_{2}-\lambda_{1})t^{'}}\int_a^{t'}e^{- \lambda_{2}t^{''}}\left[u_{t^{''}t^{''}}(X,t^{''})-(\lambda_1 + \lambda_2)\; u_{t^{''}}(X,t^{''})  + \lambda_1 \lambda_2 \; u(X,t^{''})\right] \ dt'' \ dt' \nonumber \\
&= e^{\lambda_{1}t}\int_a^t \left[ e^{-\lambda_{1}t^{'}}\left( u_{t^{'}}(X,t^{'})-\lambda_{1} u(X,t^{'}) \right) -e^{-\lambda_{2} a}e^{(\lambda_{2}-\lambda_{1}) t^{'}} \left( u_{t^{'}}(X,a)-\lambda_{1} u(X,a)\right) \right]dt^{'}\nonumber\\
 &= u(X,t) -u(X,a) e^{\lambda_1(t-a)} - \frac{u_t(X,a)-\lambda_{1}
 u(X,a)}{\lambda_2-\lambda_1}\{e^{\lambda_2(t-a)}-e^{\lambda_1(t-a)}\}. \ \ \ \ \ \ \ \ \ \ \ \ \ \ \ \ \ 
\end{align}
Thus, operating $\hat{{\cal O}}^{-1}$ on both sides of (\ref{eq2p11}) followed by some algebraic rearrangements, the solutions of Eq.(\ref{eq2p10}) can be written (for $\lambda_{1} \ne \lambda_{2} $) in the form
\begin{eqnarray}\label{eq2p15}
u(X,t) & = &  u(X,a) e^{\lambda_1(t-a)} + \frac{u_t(X,a)-\lambda_{1}
 u(X,a)}{\lambda_2-\lambda_1}\{e^{\lambda_2(t-a)}-e^{\lambda_1(t-a)} \} \nonumber \\
       & &  \ \ \ \  + \ \hat{{\cal O}}^{-1}[{\cal N}[u,\cdots u_{x_i}\cdots,\cdots u_{x_i x_j}\cdots,\cdots]](X,t)+\hat{{\cal O}}^{-1}[S](X,t).
\end{eqnarray}
To write this solution in a convenient form we use the symbol $u_{0}(X,t)$  to represent the leading order term given by
\begin{equation}\label{eq2p16}
 u_{0}(X,t) =  u(X,a) e^{\lambda_1(t-a)} + \frac{u_t(X,a)-\lambda_{1}
 u(X,a)}{\lambda_2-\lambda_1}\{e^{\lambda_2(t-a)}-e^{\lambda_1(t-a)} \}+\hat{{\cal O}}^{-1}[S](X,t).
\end{equation}
Then the solution in \eqref{eq2p15} can be put into the form
\begin{eqnarray}
u(X,t) & = & u_{0}(X,t) + \hat{{\cal O}}^{-1}[{\cal N}[u,\cdots u_{x_i} \cdots,\cdots u_{x_i x_j}\cdots,\cdots]](X,t).
\end{eqnarray}
One can now follow successive steps of ADM  to get correction terms corresponding to nonlinear operator \\ ${\cal N}[u,\cdots u_{x_i}\cdots,\cdots u_{x_i x_j}\cdots,\cdots](X,t)$ recursively by using the formula
\begin{equation}\label{eq2p17}
  u_{n}(X,t) = \hat{\cal O}^{-1}[\bar{\cal A}_{n-1}](X,t),  \ \ n \geq 1.
\end{equation}
Here we have used an alternative form of the Adomian polynomials $\bar{\cal A}_{n}(X,t)$ suggested in Ref.\citep{el2007error} 
\begin{subequations}
\begin{eqnarray}
\bar{\cal A}_{0}(X,t) &=& {\cal N}\left[u_0,\cdots {u_0}_{x_i} \cdots,\cdots {u_0}_{x_i x_j}\cdots,\cdots\right](X,t) \label{eq2p18a}\\
\bar{\cal A}_{m}(X,t) &=& {\cal N}\left[\sum_{k=0}^m u_k,\cdots \sum_{k=0}^m {u_k}_{x_i} \cdots,\cdots \sum_{k=0}^m {u_k}_{x_i x_j}\cdots,\cdots\right](X,t)-\sum_{k=0}^{m-1}\bar{\cal A}_{k}(X,t), \ \ m \geq 1 \label{eq2p18b}
\end{eqnarray}
\end{subequations}
involving $m^\textrm{th}$ partial sum $\sum_{k=0}^m {u_k}$ instead of conventional one defined by formula \eqref{eq2p7}. 

Whenever the parameters (in the exponent) $\lambda_{1},\lambda_{2}$ are equal in magnitude but of opposite sign, the operator $\hat{{\cal O}}$ and the solution becomes simpler. Their expressions involving the operator
$ \ \hat{{\cal O}}[\cdot](X,t) = (\frac{\partial^{2}}{\partial t^{2}} -\lambda^{2}) [\cdot](X,t) = (\frac{\partial}{\partial t}+\lambda)(\frac{\partial}{\partial t}-\lambda)[\cdot](X,t) $ is found to be
\begin{equation}\label{eq2p18}
  u(X,t) = u_0(X,t) +\hat{{\cal O}}^{-1}\left[{\cal N}[u,\cdots u_{x_i}\cdots,\cdots u_{x_i x_j}\cdots,\cdots]\right](X,t),
\end{equation}
\begin{equation}\label{eq2p19}
  u_{0}(X,t)= u(X,a) e^{\lambda(t-a)} + \frac{u_t(X,a)-\lambda \;
 u(X,a)}{2 \lambda}\{e^{\lambda(t-a)}-e^{-\lambda (t-a)} \}+\hat{{\cal O}}^{-1}[S](X,t),
\end{equation}
\begin{equation}\label{eq2p20}
  u_{n+1}(X,t) = \hat{\cal O}^{-1}[\bar{\cal A}_{n}](X,t),  \ \ n\geq 0.
\end{equation}
\section{Convergence Analysis}\label{sec4}
The convergence of the Adomian series solution for a variety of nonlinear terms involving unknown solution only have been exercised earlier\citep{hosseini2006convergence,abbaoui1994convergence,el2011convergence,el2007error,el2011error}. Among those, the approaches followed in Refs. \citep{el2011convergence,el2007error,el2011error} are seem to be user friendly.  Here we have established the convergence of the approximate solution of Eq. \eqref{eq2p10} whenever the nonlinear term is an appropriate section of jet space. Here it is assumed that $t \in J=[a,T], \ T\in R^+, \ \Omega=R^n \times J, \ u\in C[ \Omega]$, $S(X,t)$ is assumed to be bounded $ \forall \, (X,t) \in \Omega$ and $ \forall \ a\leq t \leq t' \leq t''\leq T $.
\begin{defn}\label{def1} The nonlinear term ${\cal N}\left[u,\cdots u_{x_i}\cdots,\cdots u_{x_i x_j}\cdots,\cdots \right]$  satisfy Lipschitz condition \citep{singh2020solving} if
\begin{eqnarray}\label{eq3p1}
 \left|  {\cal N} \left[u, \cdots u_{x_i}\cdots, \cdots u_{x_i x_j}\cdots ,\cdots \right](X,t)-{\cal N}\left[u^*,\cdots u_{x_i}^* \cdots, \cdots u_{x_i x_j}^*\cdots ,\cdots \right](X,t)\right|  \nonumber \\
 \leq L \max\limits_{(X,t) \in \Omega} \left\{ \left|u-u^* \right|,\cdots \left|u_{x_i}-u_{x_i}^* \right| \cdots, \cdots \left|u_{x_i,x_j}-u_{x_i,x_j}^* \right|\cdots, \cdots \right\} 
\end{eqnarray}

where $0 < L < \infty $, is Lipschitz constant. 
\end{defn}
\begin{cor}\label{cor1} The Lipschitz condition can be further reduced to the form
\begin{eqnarray}\label{eq3p2}
 \left|  {\cal N} \left[u, \cdots u_{x_i}\cdots, \cdots u_{x_i x_j}\cdots ,\cdots \right](X,t)-{\cal N}\left[u^*,\cdots u_{x_i}^* \cdots, \cdots u_{x_i x_j}^*\cdots ,\cdots \right](X,t)\right| 
 \leq L_1 \max\limits_{(X,t) \in \Omega} \left|u-u^* \right| 
\end{eqnarray}
where $L_1 $ is a constant. 
\end{cor}
\begin{proof}
We define\\ 
$
r=\max\limits_{(X,t) \in \Omega} \left|u-u^* \right|,
$
and 
$
R=\max\limits_{(X,t) \in \Omega} \left\{ \left|u-u^* \right|, \cdots \left|u_{x_i}-u_{x_i}^* \right| \cdots, \cdots \left|u_{x_i x_j}-u_{x_i x_j}^* \right|\cdots,\cdots \right\}.
$\\
Then, $0 < r \leq R < \infty$.
Consequently, the inequality in definition \ref{def1} can be recast as
\begin{align}
 \left|  {\cal N} \left[u, \cdots u_{x_i}\cdots, \cdots u_{x_i x_j}\cdots ,\cdots \right](X,t)- {\cal N}\left[u^*,\cdots u_{x_i}^* \cdots, \cdots u_{x_i x_j}^*\cdots ,\cdots \right](X,t)\right| 
 \leq \frac{L R}{r} r 
 = L_1 \max\limits_{(X,t) \in \Omega} \left|u-u^* \right|,\nonumber
\end{align}
where $L_1=\frac{L R}{r}.$
\end{proof}
\begin{thm}\label{th1}
Whenever ${\cal N} \left[u, \cdots u_{x_i}\cdots, \cdots u_{x_i x_j}\cdots ,\cdots \right]$ in Eq.(\ref{eq2p10}) satisfies Lipschitz condition with the Lipschitz constant $L$, the operator $\hat{\textbf{O}}^{-1}\left[{\cal N}\left[u,\cdots u_{x_i}\cdots,\cdots u_{x_i,x_j}\cdots ,\cdots\right]\right]$ is a contraction map for  $0 < L \frac{\lambda _2 \left(e^{\lambda _1 (T-a)}-1\right)-\lambda _1 \left(e^{\lambda _2 (T-a)}-1\right)}{\lambda _1 \left(\lambda _1-\lambda _2\right) \lambda _2}<1$ and there exists a unique solution to the problem (\ref{eq2p10}).
\end{thm}
\begin{proof}
We use the symbol $E=(C[\Omega]\times C[\Omega] \times C[\Omega]\times \cdots,\left\|.\right\|)$, to represent the Banach space of all continuous functions on $\Omega$ with the norm
\begin{equation}
\left\|v\right\|\equiv\left\|(\cdots, v_k(X,t)\cdots,\cdots)\right\|=\max\limits_{(X,t) \in \Omega}  \{\left|v_1(X,t)\right|,\cdots\left|v_k(X,t)\right|,\cdots\}.\nonumber
\end{equation}
In the subsequent step, we introduce a mapping $F: E\rightarrow \mathbb{R},$  defined by 
\begin{eqnarray}
F[v](X,t)=  
\theta(X,t)+\hat{\textbf{ O}}^{-1}\left[{\cal N}[v_1,\cdots, v_k, \cdots ]\right].
\end{eqnarray}
Then, for $(u,\cdots u_{x_i}\cdots,\cdots u_{x_i,x_j}\cdots ,\cdots), (u^*,\cdots u_{x_i}^* \cdots, \cdots u_{x_i x_j}^*\cdots ,\cdots)\in E$ 
\begin{align}\label{eq3p2}
& \left\|F\left[u,\cdots u_{x_i}\cdots,\cdots u_{x_i,x_j}\cdots ,\cdots \right]-F\left[u^*,\cdots u_{x_i}^* \cdots, \cdots u_{x_i x_j}^*\cdots ,\cdots \right] \right\|\hspace{2.25in} \nonumber \\
& =\max\limits_{(X,t) \in \Omega} \left|\hat{\textbf{O}}^{-1}\left[{\cal N}\left[u,\cdots u_{x_i}\cdots,\cdots u_{x_i,x_j}\cdots ,\cdots\right]  - {\cal N}\left[u^*,\cdots u_{x_i}^* \cdots, \cdots u_{x_i x_j}^*\cdots ,\cdots \right]\right]\right| \hspace{1in}\nonumber \\
& \leq \max\limits_{(X,t) \in \Omega} \hat{\textbf{O}}^{-1} \left| \left[ {\cal N}\left[u,\cdots u_{x_i}\cdots,\cdots u_{x_i,x_j}\cdots ,\cdots \right] - {\cal N}\left[u^*,\cdots u_{x_i}^* \cdots, \cdots u_{x_i x_j}^*\cdots ,\cdots \right]\right]\right| \hspace{1.1in}  \nonumber \\
& \leq  L \max\limits_{(X,t) \in \Omega} \left\{  e^{\lambda_{1}t}\int_a^t e^{(\lambda_{2}-\lambda_{1})t^{'}}\int_a^{t'}e^{-\lambda_{2}t^{''}} \ dt'' \ dt' \right\} \times \hspace{3.2in}\nonumber  \\
&  \hspace{1.2in} \max\limits_{(X,t) \in \Omega} \left\{ \left|u-u^* \right|,\cdots\left|u_{x_i}-u_{x_i}^* \right| \cdots,\,\cdots\left|u_{x_i x_j}-u_{x_i x_j}^* \right| \cdots, \cdots\right\}\nonumber \\
& \leq L \frac{\lambda _2 \left(e^{\lambda _1 (T-a)}-1\right)-\lambda _1 \left(e^{\lambda _2 (T-a)}-1\right)}{\lambda _1 \left(\lambda _1-\lambda _2\right) \lambda _2} \times \hspace{3.2in}\nonumber  \\
&  \hspace{1.2in} \max\limits_{(X,t) \in \Omega}   \left\{ \left|u-u^* \right|,\cdots\left|u_{x_i}-u_{x_i}^* \right| \cdots,\,\cdots\left|u_{x_i x_j}-u_{x_i x_j}^* \right| \cdots, \cdots\right\} \hspace{0in}  \nonumber  \\
& \leq  L \frac{\lambda _2 \left(e^{\lambda _1 (T-a)}-1\right)-\lambda _1 \left(e^{\lambda _2 (T-a)}-1\right)}{\lambda _1 \left(\lambda _1-\lambda _2\right) \lambda _2} \times \hspace{3.2in}\nonumber  \\
& \hspace{1.2in} \left\| \left( \left|u-u^* \right|,\cdots\left|u_{x_i}-u_{x_i}^* \right| \cdots,\,\cdots\left|u_{x_i x_j}-u_{x_i x_j}^* \right| \cdots, \cdots\right) \right\|.\hspace{2.55in} 
\end{align}
For the condition $0 < L \frac{\lambda _2 \left(e^{\lambda _1 (T-a)}-1\right)-\lambda _1 \left(e^{\lambda _2 (T-a)}-1\right)}{\lambda _1 \left(\lambda _1-\lambda _2\right) \lambda _2} < 1,$ the mapping $F$ i.e., $\hat{\textbf{O}}^{-1}$ (for fixed $\theta(X,t)$) is a contraction mapping for $\left(u,\cdots u_{x_i}\cdots,\cdots u_{x_i,x_j}\cdots ,\cdots\right) \in E$. Therefore, by the Banach fixed-point theorem for contraction, there exists an unique solution to the problem (\ref{eq2p10}). This completes the proof. 
\end{proof}
\begin{thm}\label{th3}For $S_p(X,t)= \sum_{i=0}^p u_i(X,t)$,
\begin{eqnarray}
\left|\sum_{i=q}^{p-1} \bar{\cal A}_i(X,t)\right|  = \left| {\cal N}[S_{p-1},\cdots (S_{p-1})_{x_i} \cdots,\cdots (S_{p-1})_{x_i x_j} \cdots,\cdots](X,t) \right. \hspace{1in} \nonumber \\ 
\hspace{2in}\left.- {\cal N}[S_{q-1},\cdots (S_{q-1})_{x_i} \cdots,\cdots (S_{q-1})_{x_i x_j} \cdots,\cdots](X,t)\right| 
\end{eqnarray}
\end{thm}
\begin{proof}
For $p,q \in \mathbb{N}$ with $p>q\geq 1$, definition \eqref{eq2p18b} leading to
\begin{equation}
\sum_{i=0}^{p-1}\bar{\cal A}_i(X,t) =  {\cal N}[S_{p-1},\cdots (S_{p-1})_{x_i} \cdots,\cdots (S_{p-1})_{x_i x_j} \cdots,\cdots](X,t)\nonumber
\end{equation}
and
\begin{equation}
\sum_{i=0}^{q-1}\bar{\cal A}_i(X,t) =  {\cal N}[S_{q-1},\cdots (S_{q-1})_{x_i} \cdots,\cdots (S_{q-1})_{x_i x_j} \cdots,\cdots](X,t).\nonumber
\end{equation}
The subtraction of relations mentioned above give the desired result  of the theorem.
\end{proof}
\begin{thm}\label{th4}
 Whenever the PDE \eqref{eq2p10} is integrable, the sequence $\{S_p(X,t), p \in \mathbb{N}\}$ of the approximate solution of the equation obtained by using RCAS of Sec. \ref{sec3} is convergent if $\left|u_1(X,t)\right|< K,$ for some $K \in \mathbb{R}^+$.
\end{thm}
\begin{proof}
To prove this we will show that the sequence $\{S_p(X,t), p \in \mathbb{N}\}$ is a Cauchy sequence. For $p,q \in \mathbb{N} (p>q\geq1),$
\begin{eqnarray}
\left| S_p(X,t) -S_q(X,t) \right| &=& \left| \sum_{i=0}^p u_i(X,t)-\sum_{i=0}^q u_i(X,t) \right|\nonumber \\
                        &=& \left| \sum_{i=q+1}^p u_i(X,t)\right| \nonumber\\
                        &=&\left| \hat{\textbf{O}}^{-1}\left[\sum_{i=q+1}^p \bar{\cal A}_{i-1}(X,t) \right] \right|  \ (\textrm{using} \eqref{eq2p20}) \nonumber \\
                        & \leq & \hat{\textbf{O}}^{-1}\left[\left|\sum_{i=q}^{p-1} \bar{\cal A}_i(X,t)\right|\right]. 
\end{eqnarray}
Use of result in theorem \ref{th3} to the above relation yields
\begin{equation}\label{eq3p3a}
\begin{split}
\left| S_p(X,t) -S_q(X,t) \right| \leq \hat{\textbf{O}}^{-1}\left[\left| {\cal N}\left[S_{p-1},\cdots (S_{p-1})_{x_i}\cdots ,\cdots(S_{p-1})_{x_i x_j}\cdots ,\cdots\right](X,t) \right.\right. \hspace{1in} \nonumber \\
\left.\left.  - {\cal N}\left[S_{q-1},\cdots(S_{q-1})_{x_i}\cdots,\cdots (S_{q-1})_{x_i x_j}\cdots, \cdots\right](X,t)\right|\right]. 
\end{split}
\end{equation}
Further use of the result in corollary \ref{cor1} gives
\begin{equation}\label{eq3p3b}
\begin{split}
\left| S_p(X,t) -S_q(X,t) \right| \leq L_1 \hat{\textbf{O}}^{-1} \left\{ \max\limits_{(X,t) \in \Omega} \left|S_{p-1}(X,t)-S_{q-1}(X,t) \right| \right\}.
\end{split}
\end{equation}
Thus,
\begin{equation}\label{eq3p4}
\begin{split}
\left\|S_p(X,t)-S_q(X,t)\right\|
& \leq  L_1 \max\limits_{(X,t) \in \Omega} \left\{e^{\lambda_{1}t}\int_a^t e^{(\lambda_{2}-\lambda_{1})t^{'}}\int_a^{t'}e^{-\lambda_{2}t^{''}} \ dt'' \ dt' \right\} \max\limits_{(X,t) \in \Omega} \left|S_{p-1}(X,t)-S_{q-1}(X,t) \right|  \\
& \leq   L_1 \frac{\lambda _2 \left(e^{\lambda _1 (T-a)}-1\right)-\lambda _1 \left(e^{\lambda _2 (T-a)}-1\right)}{\lambda _1 \left(\lambda _1-\lambda _2\right) \lambda _2}\left\|S_{p-1}(X,t)- S_{q-1}(X,t)\right\|  \\
& \leq  \alpha \left\|S_{p-1}(X,t)- S_{q-1}(X,t)\right\|,
\end{split}
\end{equation}
where $$
\alpha =L_1 \frac{\lambda _2 \left(e^{\lambda_1 (T-a)}-1\right)-\lambda _1 \left(e^{\lambda _2 (T-a)}-1\right)}{\lambda _1 \left(\lambda _1-\lambda _2\right) \lambda _2}.$$ 
For $p=q+1$,
\begin{equation}
\left\|S_{q+1}- S_{q}\right\| \leq \alpha \left\|S_{q}- S_{q-1}\right\| \leq \alpha^2 \left\|S_{q-1}- S_{q-2}\right\| \leq .... \leq \alpha^q \left\|S_{1}- S_{0}\right\|.
\end{equation}
\\
Then with the help of the triangle inequality we have 
\begin{equation}\label{eq3p5}
\begin{split}
\left\|S_{q+1}- S_{q}\right\| & \leq  \left\|S_{q+1}- S_{q}\right\|+\left\|S_{q+2}- S_{q+1}\right\|+...+\left\|S_{p}- S_{p-1}\right\| \\
& \leq  [\alpha^q + \alpha^{q+1} +...+\alpha^{p-1}] \left\|S_{1}- S_{0}\right\|  \\
& \leq  \alpha^q[1 + \alpha +...+\alpha^{p-q-1}] \left\|S_{1}- S_{0}\right\|  \\
& \leq  \alpha^q[\frac{1 - \alpha^{p-q}}{1-\alpha}] \left\|u_1\right\|.
\end{split}
\end{equation}
For $0 < \alpha < 1$, $(1- \alpha^{p-q})\leq 1.$
Consequently,
\begin{equation}\label{eq3p6}
\begin{split}
\left\|S_{p}- S_{q}\right\| & \leq  \frac{\alpha^q}{1-\alpha} \left\|u_1\right\|   \\
& \leq  \frac{\alpha^q}{1-\alpha} \max\limits_{(X,t) \in \Omega} \left|u_1\right|.
\end{split}
\end{equation}
Whenever $\left|u_1\right| < K < \infty$, $\left\|S_{p}- S_{q}\right\|\rightarrow 0$ as $q\rightarrow \infty$. Hence $\{S_p, p \in \mathbb{N} \}$ is a Cauchy sequence in the Banach space $E$ and the series$\sum_{k=0}^\infty u_k(X,t)$ is convergent.
\end{proof}
\begin{thm}\label{th5}
 The maximum absolute truncation error of the series solution to the equation (\ref{eq2p10}) is estimated to be, 
$  \max\limits_{t\in J}\left|u(X,t)- \sum_{i=0}^q u_i(X,t)\right|\leq \frac{M \alpha^{q+1}}{L_1(1-\alpha)}$ 
where $M=\max\limits_{(X,t) \in \Omega} \left|{\cal N}\left(u_{0},\cdots{u_{0}}_{x_i}\cdots, \cdots {u_{0}}_{x_i x_j}\cdots,\cdots  \right)\right|.$  
\end{thm}
\begin{proof}
In the limit $p \rightarrow \infty$, relation \eqref{eq3p6} of theorem  \ref{th4} may be put in the form
\begin{equation}\label{eq3p7}
\max\limits_{(X,t) \in \Omega} \left|u(X,t)- \sum_{i=0}^q u_i(X,t)\right|  \leq  \frac{\alpha^q}{1-\alpha} \max\limits_{(X,t) \in \Omega} \left|u_1(X,t)\right|.
\end{equation}
Since
\begin{align}\label{eq3p8}
u_1(X,t)&= e^{\lambda_{1}t}\int_a^t e^{(\lambda_{2}-\lambda_{1})t^{'}}\int_a^{t'}e^{-\lambda_{2}t^{''}} \bar{\cal A}_0(X,t) \ dt'' \ dt' ,  \nonumber \\
\max\limits_{(X,t) \in \Omega} \left|u_1(X,t)\right| & = \max\limits_{(X,t) \in \Omega} \left|e^{\lambda_{1}t}\int_a^t e^{(\lambda_{2}-\lambda_{1})t^{'}}\int_a^{t'}e^{-\lambda_{2}t^{''}} \bar{\cal A}_0(X,t) \ dt'' \ dt'\right|  \nonumber  \\
&= \max\limits_{(X,t) \in \Omega} \left|e^{\lambda_{1}t}\int_a^t e^{(\lambda_{2}-\lambda_{1})t^{'}}\int_a^{t'}e^{-\lambda_{2}t^{''}} {\cal N}\left(u_{0},\cdots{u_{0}}_{x_i}\cdots,\right. \right. \nonumber\\
& \left. \left. \hspace{1.7in} \cdots {u_{0}}_{x_i x_j}\cdots,\cdots \right) \ dt'' \ dt'\right|  \nonumber  \\
&\leq  M \max\limits_{(X,t) \in \Omega} \left|e^{\lambda_{1}t}\int_a^t e^{(\lambda_{2}-\lambda_{1})t^{'}}\int_a^{t'}e^{-\lambda_{2}t^{''}}  \ dt'' \ dt'\right| \nonumber \\
&\leq  M \frac{\lambda _2 \left(e^{\lambda _1 (T-a)}-1\right)-\lambda _1 \left(e^{\lambda _2 (T-a)}-1\right)}{\lambda _1 \left(\lambda _1-\lambda _2\right) \lambda _2}.
\end{align}
Combining the estimates (\ref{eq3p7}) and (\ref{eq3p8}) one gets
\begin{equation}\label{eq3p9}
\max\limits_{(X,t) \in \Omega} \left|u(X,t)- \sum_{i=0}^q u_i(X,t)\right|\leq \frac{M \alpha^{q+1}}{L_1 (1-\alpha)}.
\end{equation}
This completes the proof.
\end{proof}
\section{Illustrative examples}\label{sec5}
To exhibit the efficiency of RCAS presented in section \ref{sec3}, we now apply recursion formula (\ref{eq2p19}) and (\ref{eq2p20})  in the case of initial value problems for getting approximate/exact solution of the following physically interesting nonlinear PDEs.
\begin{exmp}\label{exmp1}
\end{exmp}
 We first consider the equation \citep{nourazar2013new}
\begin{equation}\label{eq7p1}
 u_{t}+\frac{1}{2}(u^2)_x =u(1-u), \ 0\leq x \leq1, 0\leq t
\end{equation}
with initial condition $u(x,0)=e^{-x}$ appear as a mathematical model in the studies of homogeneous gas dynamics. It can be verified that the exact solution is $u(x,t)=e^{t-x}.$

To test the of efficiency of the proposed approximation scheme, we write  Eq. \eqref{eq7p1} in the standard form (Eq. \eqref{eq2p11} of section 3)
\begin{equation}\label{eq7p2}
\hat{\textbf{O}}[u](x,t)={\cal N}[u,u_x](x,t),
\end{equation}
where the linear operator $\hat{\textbf{O}}[\cdot]  = e^{- t}\frac{\partial}{\partial t} (e^{t}[\cdot] )$ is a first order differential operator and the collection of nonlinear terms $\textbf{ N} [u,u_x](x,t)=-\frac{1}{2}(u^2)_x-u^2$. The inverse of the operator $\hat{\textbf{O}}^{-1}[\cdot] $ is given by
\begin{equation}\label{eq7p3}
\hat{\textbf{O}}^{-1}[\cdot]  = e^{t}\int_a^t e^{-t^{'}}[\cdot](X,t') \ dt'.
\end{equation}
Operating $\hat{\textbf{O}}^{-1}$ mentioned above on both side of (\ref{eq7p2}), the formal expression for approximate solution involving initial conditions can found as (from (\ref{eq2p15}))
\begin{equation}\label{eq7p4}
u(x,t)= u(x,0) e^{t} +\hat{\textbf{O}}^{-1}[{\cal N}[u,u_x]](x,t).
\end{equation}
We now use formulae (\ref{eq2p19}) and (\ref{eq2p20}) to obtain the leading order term and recursion relation for corrections as
\begin{equation}\label{eq7p5}
  u_{0}(x,t)=u(x,0) e^{t} = e^{-x+t}  ,
\end{equation}
\begin{equation}\label{eq7p6}
  u_{n+1}(x,t) = \hat{\textbf{O}}^{-1}[\bar{\cal A}_{n}](x,t),  \ \ n\geq 0.
\end{equation}
Here $\bar{\cal A}_{n}(x,t)$ are Adomian polynomials in its revised form. For the nonlinear term $-\frac{1}{2}(u^2)_x-u^2$, the formulae \eqref{eq2p18a},\eqref{eq2p18b} have been used for their determination. Explicit expression for first few Adomian polynomials are given by \\
\begin{eqnarray}\label{eq7p6b}
\bar{\cal A}_0(x,t)&=&- u_0^2- u_0 (u_0)_{x},\nonumber \\	
\bar{\cal A}_1(x,t)&=&-u_0 \left\{2 u_1+(u_1)_{x}\right\}-u_1 \left\{u_1+(u_0)_{x}+(u_1)_{x}\right\},\nonumber \\
&\vdots&
\end{eqnarray}
etc. Then using the expressions for leading order solution given in \eqref{eq7p5} and Adomian polynomials  mentioned above in \eqref{eq7p6} we obtain the first few corrections
\begin{eqnarray}\label{{eq7p8}}
 u_{1}(x,t) &=& e^{t}\int_0^t e^{-t^{'}} \bar{\cal A}_{0}(x,t')  \ dt'=0,\nonumber \\
 u_{2}(x,t) &=& e^{t}\int_0^t e^{-t^{'}}\bar{\cal A}_{1}(x,t') \ dt'=0, \nonumber \\
 &\vdots&
\end{eqnarray}
So, the approximate solution of the Eq. \eqref{eq7p1} 
\begin{equation}\label{eq7p10}
u(x,t)= \sum_{k=0}^\infty u_k(x,t)=e^{t-x},
\end{equation}
converges to the exact solution to the problem.

In their studies \citep{nourazar2013new}, Ghoreishi et al. have solved this equation using  FTADM  and obtained the solution in the form
\begin{equation}\label{eq7p11}
  u_{0}(x,t)=e^{-x} ,
\end{equation}
\begin{equation}\label{eq7p12}
  u_{1}(x,t) = t e^{-x}, \ \  u_{2}(x,t) =  \frac{t^2}{2}e^{-x},\cdots.
\end{equation}
Hence the approximate solution found as
\begin{equation}\label{eq7p14}
u(x,t)= \sum_{k=0}^\infty u_k(x,t)=( e^{-x}  + t e^{-x}+\frac{t^2}{2}e^{-x}+ \cdots).
\end{equation}
Comparison of leading order solution and correction terms obtained by these two methods reveals that while the approximate solution obtained by Ghoreishi et al. is an (convergent) infinite series of nonvanishing terms, the present scheme provides the exact solution in the leading term. The correction terms vanish identically. Hence RCAS seems to be rapidly convergent in comparison to the FTADM of Ghoreishi et al.  \citep{nourazar2013new}.
%
\begin{exmp}\label{exap2} 
\end{exmp}
Here we consider a nonlinear equation  \citep{polyanin2012handbook} (page 455)
\begin{equation}\label{ppl1}
 u_{tt}- b\; u = a \; \left(u u_x \right)_x. 
\end{equation}
The exact solution of this equation corresponding to the initial condition
\[ 
 u(x,0)=B_1-\frac{b (c_1+x)^2}{6 a}, \  u_t(x,0)=\sqrt{\frac{2 b}{3}}  B_2
\]
may be found as 
\begin{equation}\label{ppex2exct}
u(x,t)=-\frac{b (c_1+x)^2}{6 a}+B_1 \cosh \left(\sqrt{\frac{2 b}{3}} t\right)+B_2 \sinh \left(\sqrt{\frac{2 b}{3}}  t\right).
\end{equation}
In the scheme of RCAS, the above equation can be put into the form
\begin{equation}\label{pp12}
\hat{\textbf{O}}[u](x,t)=\mathcal{N}[u,u_x,u_{xx}](x,t),
\end{equation}
where the linear operator $\hat{\textbf{O}}[\cdot]  = e^{-\sqrt{b}\;t}\frac{\partial}{\partial t} (e^{2 \sqrt{b}\;t}\frac{\partial}{\partial t}(e^{-\sqrt{b}\;t}[\cdot] ))$ and $\mathcal{N}[u,u_x,u_{xx}](x,t)=a \; \left(u\; u_x \right)_x $ is the collection of nonlinear terms of the equation. So the inverse operator for this equation is given by
\begin{equation}\label{pp13}
\hat{\textbf{O}}^{-1}[\cdot]  = e^{\sqrt{b}\;t}\int_a^t e^{-2 \sqrt{b}\;t^{'}}\int_a^{t'}e^{\sqrt{b}\;t^{''}}[\cdot](X,t'') \ dt'' \ dt'.
\end{equation}
Operating $\hat{\textbf{O}}^{-1}[\cdot]$ mentioned above on both sides of (\ref{pp12}) and we get from (\ref{eq2p15})
\begin{equation}\label{pp14}
u(x,t)= u(x,0) e^{ \sqrt{b}\;t} + \frac{u_t(x,0)- \sqrt{b}\;
 u(x,0)}{2 \; \sqrt{b}}\{e^{ \sqrt{b}\; t}-e^{- \sqrt{b}\; t} \}+\hat{\textbf{O}}^{-1}[\mathcal{N}[u,u_x,u_{xx}]](x,t).
\end{equation}
Use of representation of $\mathcal{N}[u,u_x,u_{xx}](x,t)$ in terms Adomian polynomials $\bar{\cal A}_{n}(x,t)$ defined in \eqref{eq2p18a}, \eqref{eq2p18b} provides the
recursion relation for successive correction terms as
\begin{align}\label{pp15}
  &u_{0}(x,t)= u(x,0) e^{ \sqrt{b} \ t} + \frac{u_t(x,0)- \sqrt{b} \
 u(x,0)}{2 \ \sqrt{b}}\{e^{\sqrt{b} \ t}-e^{- \sqrt{b}\ t} \},\nonumber \\
 & u_{n+1}(x,t) = \hat{\textbf{O}}^{-1}[\bar{\cal A}_{n}](x,t),  \ \ n\geq 0.
\end{align}
The explicit expressions for first few Adomian polynomials $\bar{\cal A}_{n}(x,t)$ are 
\begin{align*}
\bar{\cal A}_0(x,t)=&  a \ \left\{u_0 (u_0)_x \right\}_x , \\
\bar{\cal A}_1(x,t)=& a \left\{(u_1)_{x}^2+2 (u_0)_{x} (u_1)_{x}+u_1
   (u_0)_{xx}+u_0 (u_1)_{xx}+u_1 (u_1)_{xx}\right\},\\
   & \vdots
   \end{align*}
Use of these expressions into the relation \eqref{pp15} gives the first few terms in the series of the approximate solution as
\begin{align*}
  u_{0}(x,t)=&\frac{1}{12 a} \left[e^{-\sqrt{b} t} \left\{6 a B_1 \left(e^{2 \sqrt{b} t}+1\right)+2 \sqrt{6} a B_2 \left(e^{2 \sqrt{b} t}-1\right)-b \left(e^{2 \sqrt{b} t}+1\right) \left(c_1+x\right)^2\right\}\right],\\
   u_{1}(x,t) =&\frac{1}{216 a} \left[e^{\sqrt{b} t} \left\{e^{-3 \sqrt{b} t} \left(-6 a B_1 \left(2 e^{\sqrt{b} t}-6 e^{2 \sqrt{b} t}+e^{4 \sqrt{b} t}+1\right)-2 \sqrt{6} a B_2 \left(2 e^{\sqrt{b} t}+e^{4 \sqrt{b} t}-1\right) \right. \right. \right.\\
  & \left. \left. \left. +3 b \left(2 e^{\sqrt{b} t}-6
   e^{2 \sqrt{b} t}+e^{4 \sqrt{b} t}+1\right) \left(c_1+x\right)^2\right)-12 a B_1+4 \sqrt{6} a B_2+6 b \left(c_1+x\right)^2\right\}\right],
   \end{align*}
  \begin{align*} 
   u_{2}(x,t)  = &\frac{1}{116640 a} \left[e^{-4 \sqrt{b} t} \left\{-3 a B_1 \left(e^{\sqrt{b} t}-1\right) \left(750 e^{4 \sqrt{b} t} \left(2
   \sqrt{b} t-1\right)+73 e^{\sqrt{b} t}+633 e^{2 \sqrt{b} t}-633 e^{5 \sqrt{b} t} \right. \right. \right. \\ 
   & \left. \left. \left. -73 e^{6 \sqrt{b} t}+2
   e^{7 \sqrt{b} t}+750 e^{3 \sqrt{b} t} \left(2 \sqrt{b} t+1\right)-2\right)-2 \sqrt{6} a B_2 \left(3
   e^{5 \sqrt{b} t} \left(340 \sqrt{b} t-183\right)+45 e^{\sqrt{b} t} \right. \right. \right. \\ 
   & \left. \left. \left. +170 e^{2 \sqrt{b} t}-170 e^{6
   \sqrt{b} t}-45 e^{7 \sqrt{b} t}+e^{8 \sqrt{b} t}+3 e^{3 \sqrt{b} t} \left(340 \sqrt{b}
   t+183\right)-1\right)+9 b \left(e^{\sqrt{b} t}-1\right)   \right. \right. \\ 
   & \left. \left. \times \left(15 e^{4 \sqrt{b} t} \left(20 \sqrt{b}
   t-1\right)+14 e^{\sqrt{b} t}+174 e^{2 \sqrt{b} t}-174 e^{5 \sqrt{b} t}-14 e^{6 \sqrt{b} t}+e^{7
   \sqrt{b} t}+15 e^{3 \sqrt{b} t}    \right. \right. \right. \\ 
   & \left. \left. \left. \times \left(20 \sqrt{b} t+1\right)-1\right)
   \left(c_1+x\right)^2\right\}\right],\\
   \vdots
 \end{align*}
Successive iteration can be continued to get the approximate solution with desired order of accuracy. After n-th iteration, the approximate solution of the problem obtained as 
\begin{equation}\label{pp20}
u_n^{ap}(x,t)\equiv S_n(x,t) = \sum_{k=0}^n u_k(x,t)
\end{equation}
with the absolute error 
\begin{equation}\label{pp21}
E_n(x,t)=  \vert u(x,t)-  u_n^{ap}(x,t) \vert.
\end{equation}
\begin{table*}[h]
\begin{center}
\caption{Absolute error $E_5(x,t)$ of Example \ref{exap2} for $c_1 = .4,\ a=.5,\ b= .7,\ B_1= .9,\ B_2=.5$ and different values of $x,t$.}\label{tab1}
\begin{tabular}{|c|c|c|c|} \hline
$x$ & $t=0.1$ & $t=0.5$& $t=1$\\ \hline \hline
$-5$&$ 2.77556 \times 10^{-15}$&$9.13825 \times 10^{-12}$&$3.79589 \times 10^{-8}$\\
$-4$&$ 1.44329 \times 10^{-15}$ &$ 5.43565 \times 10^{-12}$&$2.24037 \times 10^{-8}$\\
$-3$&$ 7.77156 \times 10^{-16}$ &$2.63589 \times 10^{-12}$&$1.06424 \times 10^{-8}$\\
$-2$&$2.22045 \times 10^{-16}$&$7.38964 \times 10^{-13}$&$2.67503 \times 10^{-9}$\\
$-1$&$ 1.11022 \times 10^{-16}$&$2.54463 \times 10^{-13}$&$1.49834 \times 10^{-9}$\\
$0$&$ 0.$&$3.44835 \times 10^{-13}$&$1.87774 \times 10^{-9}$\\
$1$&$ 0$&$4.67848 \times 10^{-13}$&$1.53684 \times 10^{-9}$\\
$2$&$ 4.44089 \times 10^{-16}$& $2.18403 \times 10^{-12}$&$8.74538 \times 10^{-9}$\\
$3$& $1.11022 \times 10^{-15}$&$4.80371 \times 10^{-12}$&$1.97479 \times 10^{-8}$\\
$4$&$2.22045 \times 10^{-15}$&$8.32578 \times 10^{-12}$&$3.45444 \times 10^{-8}$\\
$5$&$ 4.21885 \times 10^{-15}$&$1.27511 \times 10^{-11}$&$5.31348 \times 10^{-8}$ \\ \hline
\end{tabular}
\end{center}
\end{table*}
\begin{figure}[h]
\begin{center}
\includegraphics[scale=1.3]{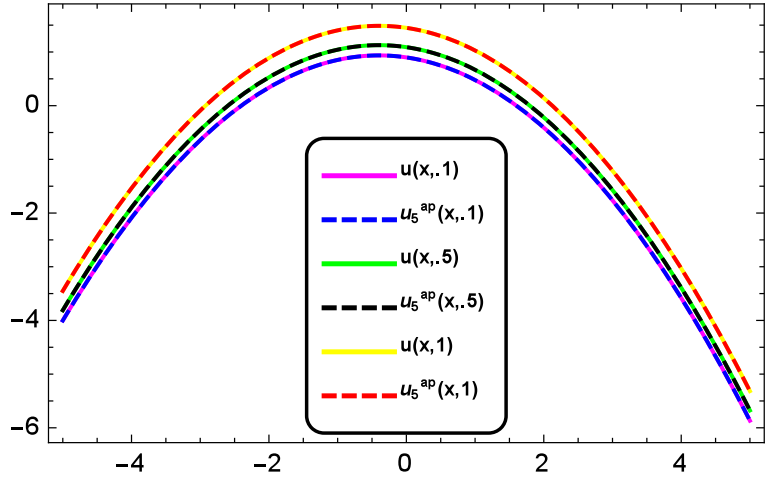}
\caption{Comparing the curves of exact ($u(x,t)$) and approximate solutions ($u_5^{ap}(x,t)$)  for $c_1 = .4,\ a=.5,\ b= .7,\ B_1= .9,\ B_2=.5$ and   different values of $t=.1,.5,1$  corresponding to Example  \ref{exap2}.}\label{fig1}
\end{center}
\end{figure}
\begin{figure*}[h]
\centering
\includegraphics[width=0.3\textwidth]{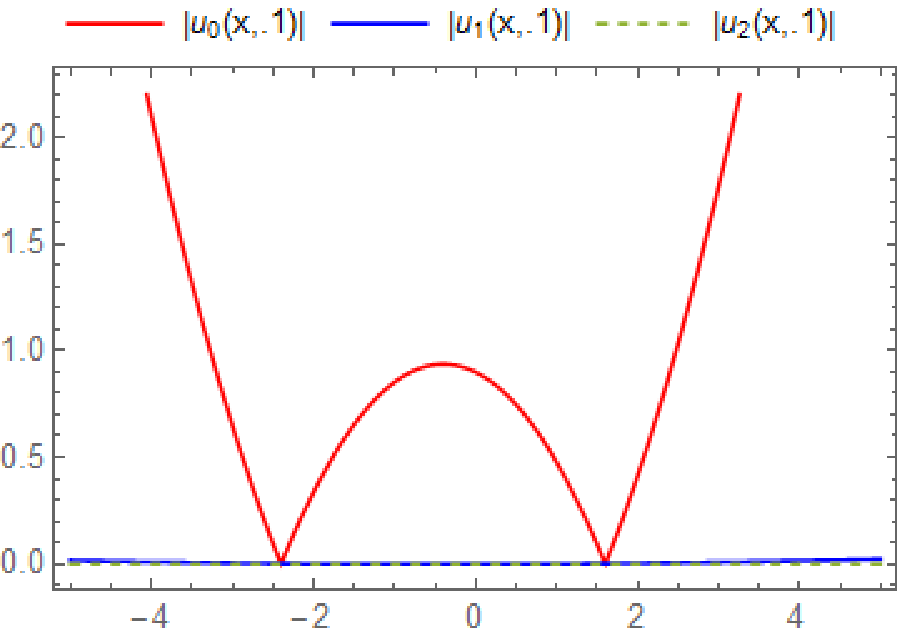}
\includegraphics[width=0.3\textwidth]{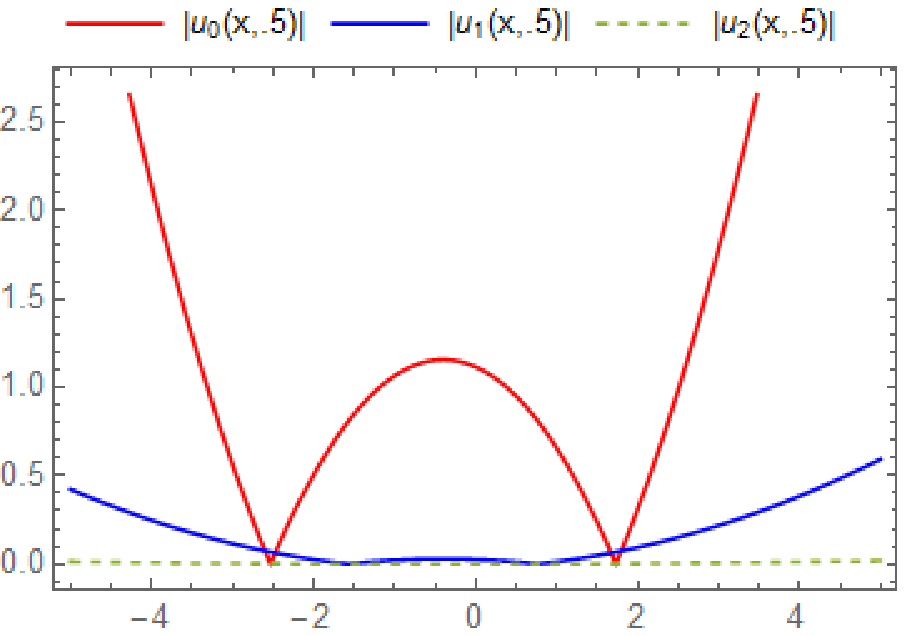}
\includegraphics[width=0.3\textwidth]{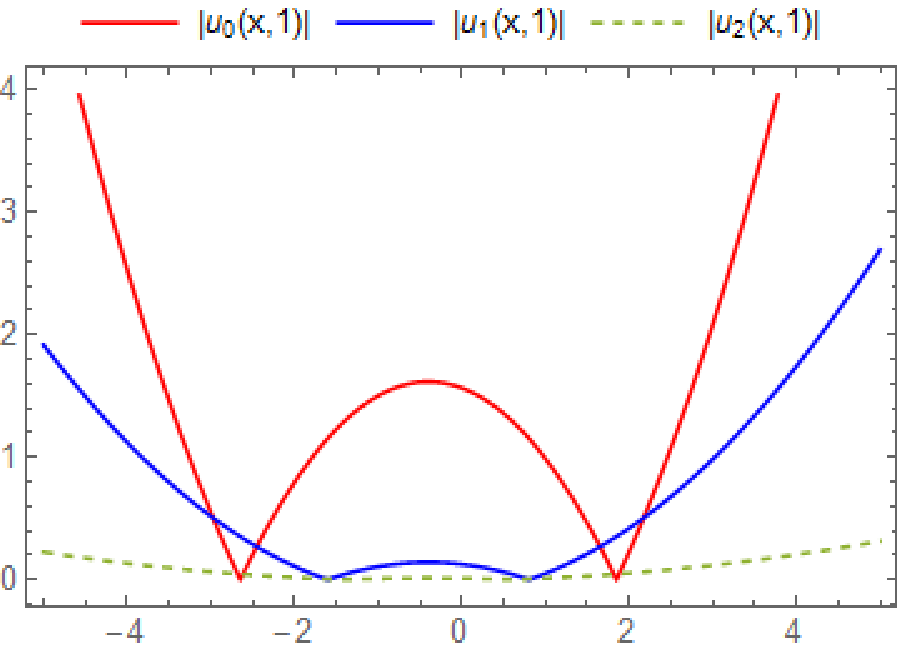}
\vskip 0.05in 
\caption{Comparison of  the correction terms $|u_0(x,t)|, \ |u_1(x,t)|,\ |u_2(x,t)|$ for $t=.1,\ .5,\  1$ corresponding to Example  \ref{exap2}.}\label{fig2}       
\end{figure*}
For illustration, the absolute errors of the approximate solution $u_5^{ap}(x,t)$ for the parameters $c_1 = .4,\ a=.5,\ b= .7,\ B_1= .9,\ B_2=.5$ for several choices of independent variables $x$ and $t$  have been presented in Table \ref{tab1}. It is found that the absolute errors are gradually increasing with $t$ but less than $O(10^{-8})$ at  $t=1$. The exact solution $u(x,t)$ given in \eqref{ppex2exct} and the approximate solution $u_5^{ap}(x,t)$ in \eqref{pp20} for $x \in (-5, 5)$ and $t = 0.1, 0.5, 1$ have been depicted in Figure \ref{fig1}. Reasonably, the figures of the exact solutions and approximate solutions appears to be indistinguishable in the region $x \in (-5, 5)$ for different values of $t$. To exhibit the rate of convergence of the approximate solution, the magnitude of leading and first two correction terms have been presented in Figure \ref{fig2} for $t = 0.1, 0.5, 1$ and found the terms are  rapidly convergent(decaying).
 \begin{exmp}
 \end{exmp}\label{exap3}
  Here we consider a generic equation with general power law nonlinearity \citep{polyanin2012handbook} (page 459)
\begin{equation}\label{pl1}
 u_{tt}- \lambda^2 u = a \ \left(u^n u_x \right)_x +b\ u^{n+1}
\end{equation}
with the initial condition 
 $$u(x,0)=\left(c_0+c_1\right) \left[B_1 \cos \left( \sqrt{\frac{b (n+1)}{a}}x\right)+B_2 \sin \left( \sqrt{\frac{b (n+1)}{a}}x\right)\right]^{\frac{1}{n+1}}$$ and $$u_t(x,0)=\left(c_0 \lambda -c_1 \lambda \right) \left[B_1 \cos \left( \sqrt{\frac{b (n+1)}{a}}x\right)+B_2 \sin \left( \sqrt{\frac{b (n+1)}{a}}x\right)\right]^{\frac{1}{n+1}}.$$
The exact solution to this initial value problem may be found as
\begin{eqnarray}\label{p11exct}
u(x,t)=\left(c_0\ e^{\lambda\ t}+c_1\  e^{- \lambda\ t}\right)\left[ B_1 \cos \left( \sqrt{\frac{b (n+1)}{a}}x\right)+B_2 \sin \left( \sqrt{\frac{b (n+1)}{a}}x\right)\right]^{\frac{1}{n+1}}.
\end{eqnarray}
To obtain an approximate solution to this equation by using the approximation scheme proposed here we write this equation as
\begin{equation}\label{p12}
\hat{\textbf{O}}[u](x,t)=\mathcal{N}[u,u_x,u_{xx}](x,t),
\end{equation}
where $\hat{\textbf{O}}[\cdot]  = e^{-\lambda\ t}\frac{\partial}{\partial t} (e^{2 \lambda\ t}\frac{\partial}{\partial t}(e^{-\lambda\ t}[\cdot] ))$ and $\mathcal{N}[u,u_x,u_{xx}](x,t)=a \ \left(u^n u_x \right)_x +b\ u^{n+1}$ is the collection of  nonlinear term of the equation. The inverse operator for this case is given by
\begin{equation}\label{p13}
\hat{\textbf{O}}^{-1}[\cdot]  = e^{\lambda\ t}\int_a^t e^{-2 \lambda\ t^{'}}\int_a^{t'}e^{\lambda\ t^{''}}[\cdot](X,t'') \ dt'' \ dt'.
\end{equation}
Application of this inverse operator on both sides of (\ref{p12}) gives
\begin{equation}\label{p14}
u(x,t)= u(x,0) e^{ \lambda\ t} + \frac{u_t(x,0)- \lambda
 u(x,0)}{2  \lambda}\{e^{ \lambda\ t}-e^{- \lambda\ t} \}+\hat{\textbf{O}}^{-1}[\mathcal{N}[u,u_x,u_{xx}]](x,t).
\end{equation}
Then following successive steps of RCAS corrections over the leading order approximation can be found as
\begin{align}\label{p15}
  &u_{0}(x,t)= u(x,0) e^{ \lambda \ t} + \frac{u_t(x,0)- \lambda \
 u(x,0)}{2 \ \lambda}\{e^{ \lambda \ t}-e^{- \lambda \ t} \},\nonumber \\
 & u_{n+1}(x,t) = \hat{\textbf{O}}^{-1}[\bar{\cal A}_{n}](x,t),  \ \ n\geq 0.
\end{align}
Here, $\bar{\cal A}_{n}(x,t)$ are Adomian polynomial for the nonlinear term $a \ \left(u^n u_x \right)_x +b\ u^{n+1}$. Using formulae \eqref{eq2p18a}, \eqref{eq2p18b} first few Adomian polynomials are calculated as  \\
\begin{align*}
\bar{\cal A}_0(x,t)=& a \ \left((u_0)^n (u_0)_x \right)_x +b\ (u_0)^{n+1}, \\
\bar{\cal A}_1(x,t)=& -(u_0)^{n-1} \left\{a n (u_0)_{x}^2+a (u_0)_{xx}
   u_0+b (u_0)^2\right\}+a n (u_0)_{x}^2
   \left(u_0+u_1\right)^{n-1}+a n (u_1)_{x}^2
   \left(u_0+u_1\right)^{n-1}\\
  & +2 a n (u_0)_{x} (u_1)_{x}
   \left(u_0+u_1\right)^{n-1}+a (u_0)_{xx}
   \left(u_0+u_1\right)^n+a (u_1)_{xx}
   \left(u_0+u_1\right)^n+b \left(u_0+u_1\right)^{n+1},\\
   & \\
   & \vdots
   \end{align*}
Then first few terms of the approximate solution may be found as
\begin{eqnarray}
  u_{0}(x,t)&=&\left(c_0\ e^{\lambda\ t}+c_1\  e^{- \lambda\ t}\right)\left[ B_1 \cos \left( \sqrt{\frac{b (n+1)}{a}}x\right)+B_2 \sin \left( \sqrt{\frac{b (n+1)}{a}}x\right)\right]^{\frac{1}{n+1}},\nonumber \\
   u_{1}(x,t) &=& 0,\nonumber \\
   u_{2}(x,t)&=& 0,\nonumber \\
  &\vdots&
 \end{eqnarray}
So, the approximate solution of the Eq. \eqref{pl1} obtained by using RCAS is
\begin{equation}\label{p20}
u(x,t)= \sum_{k=0}^\infty u_k(x,t)=\left(c_0\ e^{\lambda\ t}+c_1\  e^{- \lambda\ t}\right)\left[ B_1 \cos \left( \sqrt{\frac{b (n+1)}{a}}x\right)+B_2 \sin \left( \sqrt{\frac{b (n+1)}{a}}x\right)\right]^{\frac{1}{n+1}},
\end{equation}
the exact solution \eqref{p11exct}. It is important to note that the proposed approximation provides the exact solution in the first iteration since all the corrections $u_n(x,t)$ beginning from $n \geq 1$ vanish identically.
\begin{exmp}
\end{exmp}\label{exap4}
 Here we have considered an autonomous nonlinear partial differential equation  \citep{ghoreishi2010adomian} in (1+1) dimensions
\begin{equation}\label{equ3p1}
 u_{tt}-u = u^2\frac{\partial^2}{\partial x^2}(u_{x}u_{xx}u_{xxx})+u_x^2\frac{\partial^2}{\partial x^2}(u_{xx}^3)-18 u^5,  \  0< x < 1, \ t> 0.
\end{equation}
The initial conditions are $u(x,0)=e^{x}$ and $u_t(x,0)=e^{x}.$
It can be verified that the exact solution is $u(x,t)=e^{x+t}.$

Using the notations of RCAS, above equation an be put into the form
\begin{equation}\label{equ3p2}
\hat{\cal{O}}[u](x,t)={\cal N} [u,\cdots,u_{xxxxx}](x,t),
\end{equation}
with the linear operator
\[
 \hat{{\cal O}}[\cdot]  = e^{- t}\frac{\partial}{\partial t} (e^{2  t}\frac{\partial}{\partial t}(e^{- t}[\cdot] ))
 \]
and collection of nonlinear terms
\[
{\cal N} [u,\cdots,u_{xxxxx}](x,t)=u^2\frac{\partial^2}{\partial x^2}(u_{x}u_{xx}u_{xxx})+u_x^2\frac{\partial^2}{\partial x^2}(u_{xx}^3)-18 u^5.
\]
So, the inverse operator for this example becomes
\begin{equation}\label{equ3p3}
\hat{{\cal O}}^{-1}[\cdot]  = e^{ t}\int_a^t e^{-2  t^{'}}\int_0^{t'}e^{ t^{''}}[\cdot](x,t'') \ dt'' \ dt'.
\end{equation}
Application of this operator on both sides of (\ref{equ3p2}) in combination with  initial conditions gives exact solution in the form
\begin{equation}\label{equ3p4}
u(x,t)= u(x,0) e^{t} + \frac{u_t(x,0)-
 u(x,0)}{2 }\left(e^{ t}-e^{- t} \right)+\hat{{\cal O}}^{-1}[{\cal N}[u,\cdots,u_{xxxxx}]](x,t).
\end{equation}
We use formulae (\ref{eq2p19}) and (\ref{eq2p20}) to obtain explicit $x$ and $t$ dependence of the leading term and the formula for corrections as
\begin{equation}\label{exu2p5}
  u_{0}(x,t)=u(x,0) e^{t} + \frac{u_t(x,0)-
 u(x,0)}{2 }\{e^{ t}-e^{- t} \},
\end{equation}
\begin{equation}\label{exu2p6}
  u_{n+1}(x,t) = \hat{\cal O}^{-1}[\bar{\cal A}_{n}](x,t),  \ \ n\geq 0.
\end{equation}
Formulae \eqref{eq2p18a}, \eqref{eq2p18b} have been used for the determination of Adomian polynomials $\bar{\cal A}_{n}(x,t)$ for the nonlinear term ${\cal N} [u,\cdots,u_{xxxxx}](x,t)$. We present below explicit form for first few elements for their use in the calculation of first two correction terms as 
\begin{eqnarray}\label{equ3p6a}
\bar{\cal A}_0 (x,t)= (u_0)^2\left\{(u_0)_{x}(u_0)_{xx}(u_0)_{xxx}\right\}_{xx}+(u_0)_x^2\left\{(u_0)_{xx}^3\right\}_{xx}-18 u_0^5, 
\end{eqnarray}
\begin{align*}\label{equ3p6b}
\bar{\cal A}_1 (x,t)=&-18 u_1^5-180 u_0^3 u_1^2+\left[2 \left\{ (u_0)_{xxxx}+(u_1)_{xxxx}\right\} (u_0)_{xx}^2+\left\{3 \left\{(u_0)_{xxx}+(u_1)_{xxx}\right\}^2 \right. \right. \\
& \left. \left. +4 (u_1)_{xx}
   \left\{(u_0)_{xxxx}+(u_1)_{xxxx}\right\}+\left\{(u_0)_{x}+(u_1)_{x}\right\} \left\{(u_0)_{xxxxx}+(u_1)_{xxxxx}\right\}\right\} (u_0)_{xx}  \right.\\
   &  \left. +2 (u_1)_{xx}^2
   \left\{(u_0)_{xxxx}+(u_1)_{xxxx}\right\}+3 \left\{ (u_0)_{x}+(u_1)_{x}\right\} \left\{(u_0)_{xxx}+(u_1)_{xxx}\right\}
   \left\{(u_0)_{xxxx}+(u_1)_{xxxx}\right\} \right. \\
   & \left. +(u_1)_{xx} \left\{3 \left\{(u_0)_{xxx}+(u_1)_{xxx}\right\}^2+\left\{(u_0)_{x}+(u_1)_{x}\right\}
   \left\{(u_0)_{xxxxx}+(u_1)_{xxxxx}\right\}\right\}\right] u_1^2-90 u_0^4 u_1\\
   &+2 u_0 \left[-45 u_1^3+2 (u_0)_{xx}^2 \left\{(u_0)_{xxxx}+(u_1)_{xxxx}\right\}+2
   (u_1)_{xx}^2 \left\{(u_0)_{xxxx}+(u_1)_{xxxx}\right\} \right. \\
   & \left. +3 \left\{(u_0)_{x}+(u_1)_{x}\right\} \left\{(u_0)_{xxx}+(u_1)_{xxx}\right\}
   \left\{(u_0)_{xxxx}+(u_1)_{xxxx}\right\}+(u_1)_{xx} \left\{3 \left\{(u_0)_{xxx}+(u_1)_{xxx}\right\}^2 \right. \right.\\
   & \left. \left. +\left\{(u_0)_{x}+(u_1)_{x}\right\}
   \left\{(u_0)_{xxxxx}+(u_1)_{xxxxx}\right\}\right\}+(u_0)_{xx} \left\{3 \left\{(u_0)_{xxx}+(u_1)_{xxx}\right\}^2+4 (u_1)_{xx} \right. \right. \\
   & \left. \left.  \times
   \left\{(u_0)_{xxxx}+(u_1)_{xxxx}\right\}+\left\{(u_0)_{x}+(u_1)_{x}\right\} \left\{(u_0)_{xxxxx}+(u_1)_{xxxxx}\right\}\right\}\right] u_1+3
   \left[\left\{\left\{(u_0)_{xxxx}   \right. \right. \right.  \\
   &  \left.  \left. \left. +(u_1)_{xxxx}\right\} (u_1)_{xx}^2+2 \left\{ \left\{(u_0)_{xxx}+(u_1)_{xxx}\right\}^2+(u_0)_{xx}
   \left\{(u_0)_{xxxx}+(u_1)_{xxxx}\right\}\right\} (u_1)_{xx} \right. \right. \\
   & \left. \left. +(u_0)_{xx} \left\{2 (u_1)_{xxx}\left\{2 (u_0)_{xxx}+(u_1)_{xxx}\right\}+(u_0)_{xx}
   (u_1)_{xxxx}\right\}\right\} (u_0)_{x}^2+2 (u_1)_{x} \left\{(u_0)_{xx}+(u_1)_{xx}\right\} \right. \\
   & \left. \left\{2
   \left\{(u_0)_{xxx}+(u_1)_{xxx}\right\}^2+\left\{(u_0)_{xx}+(u_1)_{xx}\right\} \left\{(u_0)_{xxxx}+(u_1)_{xxxx}\right\}\right\} (u_0)_{x}+(u_1)_{x}^2
   \left\{(u_0)_{xx}+(u_1)_{xx}\right\} \right. \\
   & \left. \left\{2 \left\{(u_0)_{xxx}+(u_1)_{xxx}\right\}^2+\left\{(u_0)_{xx}+(u_1)_{xx}\right\}
   \left\{(u_0)_{xxxx}+(u_1)_{xxxx}\right\}\right\}\right]+u_0^2 \left[-180 u_1^3 \right. \\
   &\left.  +3 \left\{(u_0)_{x} (u_1)_{xxx}+(u_1)_{x}
   \left\{(u_0)_{xxx}+(u_1)_{xxx}\right\}\right\} (u_0)_{xxxx}+2 (u_0)_{xx}^2 (u_1)_{xxxx}+3 \left\{(u_0)_{x} \right. \right. \\
   & \left. \left.  +(u_1)_{x}\right\}
   \left\{(u_0)_{xxx}+(u_1)_{xxx}\right\} (u_1)_{xxxx}+2 (u_1)_{xx}^2 \left\{(u_0)_{xxxx}+(u_1)_{xxxx}\right\}+(u_0)_{xx} \left\{3 (u_1)_{xxx} \right. \right. \\
   & \left. \left.  \left\{2
   (u_0)_{xxx}+(u_1)_{xxx}\right\}+(u_1)_{x} (u_0)_{xxxxx}+\left\{(u_0)_{x}+(u_1)_{x}\right\} (u_1)_{xxxxx}\right\}+(u_1)_{xx} \left\{3
   \left\{(u_0)_{xxx} \right. \right. \right. \\
   & \left. \left. \left.  +(u_1)_{xxx}\right\}^2+4 (u_0)_{xx} \left\{(u_0)_{xxxx}+(u_1)_{xxxx}\right\}+\left\{(u_0)_{x}+(u_1)_{x}\right\}
   \left\{(u_0)_{xxxxx}+(u_1)_{xxxxx}\right\}\right\}\right],\\
  & \vdots
\end{align*}
 Then the use of these expressions in the formula \eqref{exu2p6} provide explicit $x,t$ dependence of the leading term and first two corrections as follows
\begin{equation}\label{equ3p7}
  u_{0}(x,t)= u(x,0) e^{t} + \frac{u_t(x,0)-
 u(x,0)}{2 }\{e^{ t}-e^{- t} \}=e^{x+t},
\end{equation}
\begin{equation}\label{equ3p8}
  u_{1}(x,t) = e^{ t}\int_0^t e^{-2  t^{'}}\int_0^{t'}e^{ t^{''}}\bar{\cal A}_{0}(x,t'') \ dt'' \ dt'=0,
\end{equation}
\begin{align}\label{equ3p9}
  u_{2}(x,t) =& e^{ t}\int_0^t e^{-2  t^{'}}\int_0^{t'}e^{ t^{''}}\bar{\cal A}_{1}(x,t'') \ dt'' \ dt'=0,\\
& \vdots \nonumber
\end{align}
It may be verified after a lengthy but straightforward calculation that higher order corrections vanish identically. So the approximate solution of the equation considered in this  example obtained by using RCAS can found as
\begin{equation}\label{equ3p10}
u(x,t)= \sum_{k=0}^\infty u_k(x,t)=e^{x+t}
\end{equation}
which seems to be the exact solution to the problem.

In their studies \citep{ghoreishi2010adomian}, M. Ghoreishi et al. solved this equation and got the leading order and first few corrections of the approximate solution in the form
\begin{equation}\label{equ3p11}
  u_{0}(x,t)= e^{x}(1+t) ,
\end{equation}
\begin{equation}\label{equ3p12}
  u_{1}(x,t) = e^{x}(\frac{t^2}{2}+\frac{t^3}{6}),
\end{equation}
\begin{align}\label{equ3p13}
  u_{2}(x,t) =& e^{x}(\frac{t^4}{24}+\frac{t^5}{120}),\\
  &\vdots \nonumber
\end{align}
Assuming that next order corrections follow the sequence of first two obtained in \citep{ghoreishi2010adomian}, the approximate solution can be put in the form
\begin{equation}\label{equ3p10}
u(x,t)= \sum_{k=0}^\infty u_k(x,t)= e^{x}(1+t+\frac{t^2}{2}+\frac{t^3}{6}+\frac{t^4}{24}+\frac{t^5}{120}- \cdots).
\end{equation}
of an infinite series seems to converge to the exact solution $ e^{x+t}$. So, for this equation RCAS proposed here seems to be efficient than the ADM exercised in \citep{ghoreishi2010adomian}.
\begin{exmp}
\end{exmp}\label{exap5}
 We have considered here the nonlinear wave-like equation \citep{ghoreishi2010adomian} with variable
coefficients in (1+1)-dimensions
\begin{equation}\label{eeq3p1}
 u_{tt}+u = x^2\left\{\frac{\partial}{\partial x }(u_{x} u_{xx})-u_{xx}^2\right\}, \ 0\leq x \leq1,\ 0\leq t
\end{equation}
with the initial condition $u(x,0)=0$ and $u_t(x,0)=x^2.$
It can be verified that the exact solution is $u(x,t)=x^2 \textrm{sin}t.$
To obtain the approximate solution of Eq. (\ref{eeq3p1}) by using RCAS proposed here, we write this equation as
\begin{equation}\label{eeq3p2}
\hat{\cal{O}}[u](x,t)={\cal N}[u,\cdots,u_{xxx}](x,t)
\end{equation}
where the linear operator $\hat{{\cal O}}[\cdot]  = e^{-i t}\frac{\partial}{\partial t} (e^{2 i t}\frac{\partial}{\partial t}(e^{-i t}[\cdot] ))$ and the collection of nonlinear term ${\cal N} [u,\cdots,u_{xxx}](x,t)=x^2\frac{\partial}{\partial x }(u_{x} u_{xx})-x^2( u_{xx})^2$. So, the inverse operator is given by
\begin{equation}\label{eeq3p3}
\hat{{\cal O}}^{-1}[\cdot]  = e^{i t}\int_a^t e^{-2 i t^{'}}\int_a^{t'}e^{i t^{''}}[\cdot](x,t'') \ dt'' \ dt'.
\end{equation}
Applying $\hat{{\cal O}}^{-1}$ on both sides of (\ref{eeq3p2}) and the initial conditions, the successive terms of approximate solution can be written formally as
\begin{equation}\label{eeq3p51}
  u_{0}(x,t)=u(x,0) e^{it} + \frac{u_t(x,0)-i
 u(x,0)}{2 i}\{e^{i t}-e^{-i t} \} = x^2 \textrm{sin}t,
\end{equation}
\begin{equation}\label{eeq3p61}
  u_{n+1}(x,t) = \hat{\cal O}^{-1}[\bar{\cal A}_{n}](x,t),  \ \ n\geq 0.
\end{equation}
The explicit expression for first few Adomian polynomials $\bar{\cal A}_{n}(x,t)$  for the nonlinear term  $ x^2\left\{\frac{\partial}{\partial x }(u_{x} u_{xx})-u_{xx}^2\right\}$ are
\begin{eqnarray}
\bar{\cal A}_0(x,t)&=&x^2 \left(u_0\right)_{x} \left(u_0\right)_{xxx}, \label{eex3p6a}\\
\nonumber \\
\bar{\cal A}_1(x,t)&=&x^2 \left(u_1\right)_{x} \left(u_0\right)_{xxx}+x^2 \left(u_0\right)_{x} \left(u_1\right)_{xxx}+x^2 \left(u_1\right)_{x} \left(u_1\right)_{xxx},\label{eex3p6b} \\
\vdots\nonumber
\end{eqnarray}
Then use of the leading term in \eqref{eeq3p51} and Adomian polynomials in \eqref{eex3p6a}, \eqref{eex3p6b} into \eqref{eeq3p61} give first two corrections of the approximation
\begin{equation}\label{eeq3p8}
  u_{1}(x,t) = e^{i t}\int_0^t e^{-2 i t^{'}}\int_0^{t'}e^{i t^{''}}\bar{\cal A}_{0}(x,t'') \ dt'' \ dt'=0,
\end{equation}
\begin{equation}\label{eeq3p9}
  u_{2}(x,t) = e^{i t}\int_0^t e^{-2 i t^{'}}\int_0^{t'}e^{i t^{''}}\bar{\cal A}_{1}(x,t'') \ dt'' \ dt'=0.
\end{equation}
It may be verified that the remaining correction terms are zero identically. So, the approximate solution of the equation obtained by the present method becomes
\begin{equation}\label{eeq3p10}
u(x,t)= \sum_{k=0}^\infty u_k(x,t)=x^2 \textrm{sin} t.
\end{equation}
Ghoreishi et al. used ADM in Ref.\citep{ghoreishi2010adomian} for the determination of solution of the same equation and obtain
\begin{equation}\label{eeq3p11}
  u_{0}(x,t)= t x^2, \ \   u_{1}(x,t) = - \frac{t^3}{3!}x^2, \ \  u_{2}(x,t) = - \frac{t^5}{5!}x^2 \cdots
\end{equation}
so that the approximate solution is given in the form
\begin{equation}\label{eeq3p10}
u(x,t)= \sum_{k=0}^\infty u_k(x,t)=(  t x^2  - \frac{t^3}{3!}x^2- \frac{t^5}{5!}x^2- \cdots).
\end{equation}
It is clear that the series of infinite terms converge to the exact solution. This observation establishes further the efficiency of the proposed RCAS for obtaining an accurate approximate solution of nonlinear autonomous as well as  non-autonomous PDEs.
\begin{exmp}
\end{exmp}\label{exap6}
 We consider here the  nonlinear wave-like equation in (2+1) dimensions with variable  \citep{ghoreishi2010adomian}
coefficients
\begin{equation}\label{exx3p1}
 u_{tt}+u =\frac{\partial^2}{\partial x \partial y}(u_{xx}\;u_{yy})-\frac{\partial^2}{\partial x \partial y}(x\; y\; u_{x}\;u_{y})
\end{equation}
The initial conditions are $u(x,y,0)=e^{x\ y}$ and $u_t(x,y,0)=e^{x\ y}.$
It can be verified that the exact solution is $u(x,y,t)=e^{x\ y}(\textrm{sin} t+\textrm{cos} t).$

Following the symbols used in RCAS, we write the given in the form
\begin{equation}\label{exx3p2}
\hat{\cal{O}}[u](x,y,t)={\cal N} [u,u_x,\cdots,u_{yyyy}](x,y,t),
\end{equation}
with $\hat{{\cal O}}[\cdot]  = e^{-i\ t}\frac{\partial}{\partial t} (e^{2\ i\ t}\frac{\partial}{\partial t}(e^{-i\ t}[\cdot] ))$ and ${\cal N} [u,u_x,\cdots,u_{yyyy}](x,y,t)=\frac{\partial^2}{\partial x \partial y}(u_{xx}u_{yy})-\frac{\partial^2}{\partial x \partial y}(x y u_{x}u_{y})$. The inverse operator in this case is given by
\begin{equation}\label{exx3p3}
\hat{{\cal O}}^{-1}[\cdot]  = e^{i\ t}\int_0^t e^{-2\ i\ t^{'}}\int_0^{t'}e^{i\ t^{''}}[\cdot](X,t'') \ dt'' \ dt'.
\end{equation}
Application of this operator on both side of (\ref{exx3p2}) and using initial conditions one gets (using (\ref{eq2p15}))
\begin{equation}\label{eqq3p4}
u(x,y,t)= u(x,y,0) e^{i\ t} + \frac{u_t(x,y,0)-i\ 
 u(x,y,0)}{2\ i}\{e^{i\ t}-e^{-i\ t} \}+\hat{{\cal O}}^{-1}[{\cal N}[u,u_x,\cdots,u_{yyyy}]](x,y,t).
\end{equation}
Then use of the formulae (\ref{eq2p19}) and (\ref{eq2p20}) the higher order corrections
\begin{equation}\label{eqq3p5}
 u_{n+1}(x,y,t) = \hat{\cal O}^{-1}[\bar{\cal A}_{n}](x,y,t),  \ \ n\geq 0
\end{equation}
with
\begin{equation}\label{eqq3p6}
u_{0}(x,y,t)=u(x,y,0) e^{i\ t} + \frac{u_t(x,y,0)-i\ 
 u(x,y,0)}{2\ i}\{e^{i\ t}-e^{-i\ t} \},
\end{equation}
where $\bar{\cal A}_{n}(x,y,t), n\geq 0$ are Adomian polynomials for the nonlinear term $\frac{\partial^2}{\partial x \partial y}(u_{xx}u_{yy})-\frac{\partial^2}{\partial x \partial y}(x y u_{x}u_{y})$ may be obtained by using formula (\ref{eq2p7}). First few Adomian polynomials  for example are 
\begin{eqnarray}
\bar{\cal A}_0(x,y,t)=\left\{(u_0)_{xx}(u_0)_{yy} \right\}_{xy}-\left\{x y (u_0)_{x}(u_0)_{y} \right\}_{xy}, \nonumber
\end{eqnarray}
    \begin{align*}
\bar{\cal A}_1(x,y,t)=&-x y (u_1)_{x,y}^2-y (u_0)_{y} (u_1)_{x,y}-x (u_0)_{x} (u_1)_{x,y}-x (u_1)_{x} (u_1)_{x,y}-2 x y (u_0)_{x,y}
   (u_1)_{x,y} \\
   & -(u_0)_{y} (u_1)_{x}-y (u_0)_{yy} (u_1)_{x}-x (u_1)_{x} (u_0)_{x,y}-x y (u_1)_{x} (u_0)_{xyy}-x y
   (u_0)_{x} (u_1)_{xyy} \\
   &-x y (u_1)_{x} (u_1)_{xyy}+(u_1)_{xyyy} (u_0)_{xx}-x (u_0)_{y} (u_1)_{xx}-x y (u_0)_{yy}
   (u_1)_{yy} +(u_0)_{xyyy} (u_1)_{xx} \\
   &+(u_1)_{xyyy} (u_1)_{xx}+(u_1)_{xyy} (u_0)_{xxy}-x y (u_0)_{y}
   (u_1)_{xxy}+(u_0)_{xyy} (u_1)_{xxy}+(u_1)_{xyy} (u_1)_{xxy} \\
   &-(u_1)_{y} \left[ (u_0)_{x}+(u_1)_{x}+x
   \left\{ (u_0)_{xx}+(u_1)_{xx}\right\}+y \left\{(u_0)_{xy}+(u_1)_{xy}+x \left\{(u_0)_{xxy}+(u_1)_{xxy}\right\}\right\} \right] \\
   &+(u_1)_{yyy}
   (u_0)_{xxx}+\left\{(u_0)_{yyy}+(u_1)_{yyy}\right\} (u_1)_{xxx}-(u_1)_{yy} \left[y \left\{(u_0)_{x}+(u_1)_{x}+x
   \left((u_0)_{xx}+(u_1)_{xx}\right)\right\} \right. \\
   & \left. -(u_0)_{xxxy}-(u_1)_{xxxy}\right]+(u_0)_{yy} (u_1)_{xxxy}\\
   \vdots
     \end{align*}
Use of last three expressions into the formulae for higher order corrections give 
\begin{equation}\label{eqq3p7}
  u_{0}(x,y,t)= u(x,y,0) e^{i\ t} + \frac{u_t(x,y,0)-i\
 u(x,y,0)}{2\ i}\{e^{i\ t}-e^{-i\ t} \}=e^{x\ y}(\textrm{sin}t+\textrm{cos}t) ,
\end{equation}
\begin{equation}\label{eqq3p8}
  u_{1}(x,y,t) = e^{i\ t}\int_0^t e^{-2\ i\ t^{'}}\int_0^{t'}e^{i\ t^{''}}\bar{\cal A}_{0}(x,y,t'') \ dt'' \ dt'=0,
\end{equation}
\begin{align}\label{eqq3p9}
  u_{2}(x,y,t) &= e^{i\ t}\int_0^t e^{-2\ i\ t^{'}}\int_0^{t'}e^{i\ t^{''}}\bar{\cal A}_{1}(x,y,t'') \ dt'' \ dt'=0,\\
 &  \vdots \nonumber
\end{align}
So, the approximate solution obtained by using RCAS becomes
\begin{equation}\label{eqq3p10}
u(x,y,t)= \sum_{k=0}^\infty u_k(x,y,t)=e^{x\ y}(\textrm{sin}t+\textrm{cos}t)
\end{equation}
seems to be the exact solution to the problem.
M. Ghoreishi et al., solve this problem in \citep{ghoreishi2010adomian} and obtain the approximate solution in the form
\begin{equation}\label{eqq3p11}
  u_{0}(x,y,t)= e^{x\ y}(1+t) ,
\end{equation}
\begin{equation}\label{eqq3p12}
  u_{1}(x,y,t) = e^{x\ y}(-\frac{t^2}{2}-\frac{t^3}{6}),
\end{equation}
\begin{align}\label{eqq3p13}
  u_{2}(x,y,t) &= e^{x\ y}(\frac{t^4}{24}+\frac{t^5}{120})\\
  & \vdots \nonumber
\end{align}
Assuming the successive corrections follow the sequence with first two corrections mentioned above, their approximate solution may be put in the form
\begin{equation}\label{eqq3p10}
u(x,y,t)= \sum_{k=0}^\infty u_k(x,y,t)= e^{x\ y}(1+t-\frac{t^2}{2}-\frac{t^3}{6}+\frac{t^4}{24}+\frac{t^5}{120}- \cdots ).
\end{equation}
It is clear that the series converges to the exact solution. It is worthy to mention here that the approximate solution obtained by using RCAS  provides the exact expression  in the leading order while the same in ADM appears as a power series in the time variable $t$ of the exact solution. Hence RCAS seems to be more efficient than the ADM as given in \citep{ghoreishi2010adomian} for this example.

\section{Conclusion}\label{sec6}
This work deals with the development of an easy to exercise approximation scheme for getting rapidly convergent approximation solution of initial value problems in many dimension within a bounded or unbounded domain. The underlying idea of the method proposed here is that the entire linear part involved in the equation has been taken in operator form, resulting in rapidly convergent corrections over the leading term in the approximate solution to the problem. The convergence of the sum of the successive correction terms has been established and an estimate of the error in the approximation has also been presented. From our study, it appears that:
\begin{itemize}
\item[i)] When the linear part of the given problem contains the highest order derivative term only (in the absence of first derivative and derivative-free term in the linear part in this study), successive steps and results from the present method are same as the corresponding results  of ADM. Thus, present scheme may be regarded as the generalization of the ADM. 
\item[ii)] In other cases  ($ \lambda \ \neq \ 0$), results obtained by using the present method are seem to be rapidly convergent and  more accurate in comparison to other approximation methods. Even the approximate solution provides the exact one in some cases (resulting from the higher order corrections are zero identically). 
\item[iii)]Not only initial value problems, RCAS is also able to provide more accurate results in case of boundary value problems too. In case of BVP in an unbounded domain, general term in the series of approximation can be speculated from the first few terms of the approximation so that the exact solution localized within a finite region can be recovered in a straightforward way.
\end{itemize}
These observations suggest that extension of the method developed here for getting an accurate approximate solution to other families of partial differential equations involving factorisable linear part with variable coefficients and to BVPs involving  system of  nonlinear PDEs seems to be worthy. Works in this direction are in progress, will be reported elsewhere.
\section*{Acknowledgements}
This work is supported in part by UGC assisted SAP(DRS Phase-III) program grant No. F.510/4/DRS/2009(SAP-I)  through Department of Mathematics, Visva-Bharati,  Santiniketan-731 235, W.B., India.
\section*{References}
\bibliographystyle{elsarticle-num}
\bibliography{wavelike}
\end{document}